\documentclass[11pt]{article}
\usepackage{amsmath,amssymb,amsthm}
\usepackage[margin=1in]{geometry}
\usepackage[hidelinks]{hyperref}
\hypersetup{pdftitle={Typical growth of the Furedi--Hajnal and Stanley--Wilf limits},pdfauthor={Jesse Geneson}}
\newtheorem{theorem}{Theorem}
\newtheorem{lemma}[theorem]{Lemma}
\newtheorem{corollary}[theorem]{Corollary}
\newcommand{\ex}{\operatorname{ex}}
\newcommand{\Prb}{\mathbb P}
\title{Typical growth of the F\"uredi--Hajnal\\and Stanley--Wilf limits}
\author{Jesse Geneson}
\date{}
\begin{document}
\maketitle

 \begin{abstract}
 \noindent
We prove that the F\"uredi--Hajnal limit and the Stanley--Wilf limit
of a uniformly random permutation matrix of order $k$ are at most
$\exp\bigl(O(\sqrt{k}(\log k)^{5/2})\bigr)$ with probability tending
to one as $k\to\infty$. This improves the bound
$\exp\bigl(O(k^{2/3}(\log k)^{7/3}/(\log\log k)^{1/3})\bigr)$
of Cibulka and Kyn\v cl. Together with the lower bound due to Fox,
these bounds show that the logarithms of both limits are
$k^{1/2+o(1)}$ for almost all permutations.
 \\[2mm]
 \end{abstract}

\section{Introduction}

We say that a zero-one matrix $A$ \emph{contains} a zero-one matrix
$P$ if some submatrix of $A$ can be turned into $P$ by changing
some ones to zeros. The selected rows and columns retain their order.
Otherwise $A$ \emph{avoids} $P$. Let $\ex(n,P)$ be the maximum
number of ones in an $n\times n$ matrix avoiding $P$. A
\emph{permutation matrix} is a square zero-one matrix with exactly
one $1$ in each row and column. F\"uredi and Hajnal \cite{FH} studied these extremal functions and
conjectured that $\ex(n,P)=O(n)$ for every permutation matrix $P$.
Marcus and Tardos \cite{MT} proved this conjecture. For a permutation
matrix $P$ of order at least two, define its \emph{F\"uredi--Hajnal
limit} by $c(P)=\lim_{n\to\infty}\ex(n,P)/n$.
Superadditivity \cite[Lemma 1(ii)]{PT} gives the existence of this
limit and $\ex(n,P)\leq c(P)n$ for every $n$. Marcus and Tardos's
theorem implies that $c(P)$ is finite.

The coefficient $c(P)$ is also related to permutation enumeration.
Let $\operatorname{Av}_n(P)$ be the set of permutation matrices
of order $n$ avoiding $P$. The Stanley--Wilf conjecture asserted
that $|\operatorname{Av}_n(P)|\leq K_P^n$ for some constant $K_P$.
Alon and Friedgut \cite{AF} obtained an almost-exponential bound
using generalized Davenport--Schinzel sequences. Klazar
\cite{Klazar} proved that the F\"uredi--Hajnal conjecture implies
the Stanley--Wilf conjecture, so Marcus and Tardos's theorem
resolved both conjectures. Arratia \cite{Arratia} proved that
$|\operatorname{Av}_n(P)|$ is supermultiplicative. Thus the
\emph{Stanley--Wilf limit}
$L(P)=\lim_{n\to\infty}|\operatorname{Av}_n(P)|^{1/n}$ exists and
is finite, with $|\operatorname{Av}_n(P)|\leq L(P)^n$ for every $n$.
Cibulka \cite{Cibulka} established the
polynomial comparisons
\begin{equation}\label{eq:comparisons}
 L(P)\leq2.88c(P)^2,
 \qquad c(P)\leq A L(P)^{9/2},
\end{equation}
where $A$ is an absolute constant. Cibulka and Kyn\v cl
\cite[Theorem 1.8]{CK} later improved the second comparison to
$c(P)=O\bigl(L(P)^{11/4}\log L(P)\bigr)$ as $L(P)\to\infty$;
\eqref{eq:comparisons} suffices for our purposes.

We consider how these limits depend on the order of a typical
forbidden permutation. Fox \cite{Fox} proved that $c(P)\leq\exp(O(k))$
for every permutation matrix $P$ of order $k$, and that both
limits are at least $\exp(k^{\Omega(1)})$ for almost all permutations.
Let $P_k$ be uniformly distributed over the $k!$ permutation
matrices of order $k$. The stronger lower bound due to Fox,
as stated by Cibulka and Kyn\v cl \cite[Section 1]{CK}, holds with
probability tending to one:
\begin{equation}\label{eq:lower}
 c(P_k)\geq\exp\bigl(\Omega(\sqrt{k/\log k})\bigr).
\end{equation}
Cibulka and Kyn\v cl
\cite[Corollary 1.2]{CK} proved the upper bound
\[
 c(P_k)\leq
 \exp\left(O\left(\frac{k^{2/3}(\log k)^{7/3}}
                              {(\log\log k)^{1/3}}\right)\right)
\]
with probability tending to one. We improve this bound as follows.
All logarithms below are natural.

\begin{theorem}\label{thm:random}
There is an absolute constant $C>0$ such that, if $P_k$ is uniformly
distributed over the $k!$ permutation matrices of order $k$, then
\[
 \Prb\left(\max\{c(P_k),L(P_k)\}\leq
        \exp\bigl(C\sqrt{k}(\log k)^{5/2}\bigr)\right)
 \longrightarrow 1\qquad\text{as }k\to\infty.
\]
\end{theorem}

Theorem~\ref{thm:random} determines the power of $k$ in the logarithms
of both limits. More precisely, combining the theorem with
\eqref{eq:comparisons} and \eqref{eq:lower} gives the following consequence.

\begin{corollary}\label{cor:exponent}
For every $\varepsilon>0$, a uniformly random permutation matrix
$P_k$ of order $k$ satisfies
\[
 k^{1/2-\varepsilon}\leq\log c(P_k)\leq k^{1/2+\varepsilon}
 \quad\text{and}\quad
 k^{1/2-\varepsilon}\leq\log L(P_k)\leq k^{1/2+\varepsilon}
\]
with probability tending to one as $k\to\infty$.
\end{corollary}

Cibulka and Kyn\v cl obtain their typical upper bound by controlling
the multiplicities of displacement vectors between pairs of ones.
We use their theorem converting a finite density bound into a bound
on $c(P)$, but prove a different embedding lemma based on cell
occupancies. Partition the rows and columns of $P$ into intervals
of length $s$, apart from possibly
shorter final intervals, and let $t_s(P)$ be the largest number of
ones in one cell. In Section~\ref{sec:cells}, we prove the
deterministic bound
\[
 \log c(P)=O\left(\frac{k\,t_s(P)}s(\log k)^2\right)
 \qquad\left(1\leq s\leq\lfloor k/2\rfloor+1\right).
\]
We place each row interval and column interval of $P$ at an
independently chosen translate in a larger matrix. The probabilities
that pattern ones land on zeros may differ, but the sums of these
probabilities over any row interval or column interval are bounded
using the total number of zeros in the corresponding strip and
the cell occupancy $t_s(P)$. A consequence of the asymmetric local
lemma then gives a copy of $P$ when each row and column
of the larger matrix has few zeros. Deleting rows and columns with
many zeros gives a finite density bound, to which we apply
Cibulka and Kyn\v cl's theorem \cite[Theorem 1.5]{CK}.
In Section~\ref{sec:random}, we show
that $s=\lceil\sqrt{k\log k}\rceil$ gives $t_s(P_k)=O(\log k)$
with probability tending to one, which proves Theorem~\ref{thm:random}.

\section{A bound from cell occupancies}\label{sec:cells}

For integers $k\geq2$ and $1\leq s\leq k$, partition $[k]=\{1,\ldots,k\}$
into the consecutive intervals
\[
 I_a=\{(a-1)s+1,\ldots,\min(as,k)\},
 \qquad 1\leq a\leq m=\lceil k/s\rceil.
\]
For a permutation matrix $P$ of order $k$, let $t_s(P)$ be the
maximum number of ones in any cell $I_a\times I_b$. Note that
$1\leq t_s(P)\leq s$.

\begin{theorem}\label{thm:cells}
There is an absolute constant $C>0$ such that, for every permutation
matrix $P$ of order $k\geq2$ and every integer
$1\leq s\leq\lfloor k/2\rfloor+1$, we have
\[
 \log c(P)\leq C\frac{k\,t_s(P)}s(\log k)^2.
\]
\end{theorem}

We use the following consequence of the asymmetric Lov\'asz local
lemma; see Moser and Tardos \cite[Theorem 1.1]{MoT}.

\begin{lemma}\label{lem:lll}
Let $E_1,\ldots,E_N$ be events determined by independent random
variables. Join two distinct events when their sets of variables
intersect, and let $N(i)$ be the set of neighbors of $E_i$.
Write $p_i=\Prb(E_i)$. If, for every $i$,
\[
 p_i\leq\frac14
 \quad\text{and}\quad
 \sum_{j\in N(i)}p_j\leq\frac14,
\]
then $\Prb(\bigcap_{i=1}^N\overline E_i)>0$.
\end{lemma}

\begin{proof}
The asymmetric local lemma gives a positive avoidance probability if
there are $x_i\in(0,1)$ such that
$p_i\leq x_i\prod_{j\in N(i)}(1-x_j)$ for every $i$.
Discard events of probability zero and restrict the graph to the
remaining events. Set $x_i=2p_i$. Since $\prod_j(1-a_j)\geq1-\sum_j a_j$ for
$0\leq a_j\leq1$, we have
\[
 x_i\prod_{j\in N(i)}(1-x_j)
 \geq2p_i\left(1-2\sum_{j\in N(i)}p_j\right)\geq p_i.
\]
Moreover, $0<x_i\leq1/2$. The asymmetric local lemma gives
$\Prb(\bigcap_i\overline E_i)\geq\prod_i(1-x_i)>0$.
\end{proof}

The next lemma gives the embedding needed for
Theorem~\ref{thm:cells}.

\begin{lemma}\label{lem:embedding}
Let $P$ be a permutation matrix of order $k\geq2$, let
$1\leq s\leq k$ be an integer, and put $m=\lceil k/s\rceil$
and $t=t_s(P)$.
Every $2sm\times2sm$ zero-one matrix with at most $s/(32t)$ zeros
in each row and each column contains $P$.
\end{lemma}

\begin{proof}
Let $A$ be such a matrix, and put $d=s/(32t)$. Partition the rows
and columns of $A$ into $m$ consecutive intervals of length $2s$.
Choose independent random variables
$X_1,\ldots,X_m,Y_1,\ldots,Y_m$, each uniform on
$\{0,\ldots,s-1\}$. For a pattern row $i=(a-1)s+u\in I_a$ and
a pattern column $j=(b-1)s+v\in I_b$, define
\[
 f(i)=2s(a-1)+u+X_a,
 \qquad
 g(j)=2s(b-1)+v+Y_b.
\]
Both maps are strictly increasing. Within an interval the shift is
constant, and every image of $I_a$ precedes every image of $I_{a+1}$.

Write $P_{i,\pi(i)}=1$, and let $E_i$ be the event that
$A_{f(i),g(\pi(i))}=0$. If $(i,\pi(i))$ lies in $I_a\times I_b$,
then $E_i$ uses only $X_a$ and $Y_b$. Thus we will apply
Lemma~\ref{lem:lll} by bounding the total event probability in each
row interval and each column interval. Let $z_{ab}$ be the number
of zeros in the corresponding $2s\times2s$ cell of $A$.
The $s^2$ choices of $(X_a,Y_b)$ send $(i,\pi(i))$ to distinct
positions, so $p_i=\Prb(E_i)\leq z_{ab}/s^2$.
There are at most $t$ ones of $P$ in each cell. Thus
\[
 \sum_{i\in I_a}p_i
 \leq\frac{t}{s^2}\sum_{b=1}^m z_{ab}
 \leq\frac{2td}s=\frac1{16},
\]
since the corresponding $2s$ rows of $A$ each have at most $d$
zeros. The same argument with columns gives
$\sum_{i:\,\pi(i)\in I_b}p_i\leq1/16$.

Join distinct events whose pattern points lie in a common row
interval or a common column interval. Every event has probability
at most $1/16$, and the sum of the probabilities of its neighbors
is at most $1/8$. Lemma~\ref{lem:lll} gives a choice of the shifts
for which no $E_i$ occurs. The increasing maps $f$ and $g$ then
give a copy of $P$ in $A$.
\end{proof}

We use the following theorem of Cibulka and Kyn\v cl
\cite[Theorem 1.5]{CK}: if $u$ is a positive integer,
$1/u<q<1$, and $P$ is a permutation matrix satisfying
$\ex(u,P)<qu^2$, then
\begin{equation}\label{eq:transfer}
 c(P)\leq2u^{3+\lceil\log u/(-\log q)\rceil}.
\end{equation}

\begin{proof}[Proof of Theorem~\ref{thm:cells}]
Put $t=t_s(P)$, $N=2s\lceil k/s\rceil$, and $B=ks/(32t)$.
Our bound on $s$ gives
\[
 N\leq2k+2s-2\leq3k.
\]
We claim that every $4k\times4k$ zero-one matrix $A$ with at most
$B$ zeros contains $P$. Delete $k$ rows with the most zeros.
Every remaining row has at most $B/(k+1)<s/(32t)$ zeros.
Next delete $k$ columns with the most zeros in the remaining
matrix. Every surviving column also has at most $B/(k+1)$ zeros,
and deleting columns does not increase the number of zeros in
any row. The resulting $3k\times3k$ matrix contains an
$N\times N$ submatrix to which Lemma~\ref{lem:embedding} applies.
This proves the claim.

Consequently every $P$-avoiding matrix of order $4k$ has more
than $B$ zeros, so
\[
 \ex(4k,P)<16k^2-B=q(4k)^2,
 \qquad q=1-\frac{s}{512tk}.
\]
Since $s\leq k$ and $t\geq1$, we have
$1/(4k)<511/512\leq q<1$. Apply \eqref{eq:transfer} with $u=4k$.
Using $-\log q\geq1-q=s/(512tk)$ and
$\lceil x\rceil\leq x+1$, we obtain
\[
 c(P)\leq2(4k)^{4+(512tk/s)\log(4k)}.
\]
Taking logarithms gives the stated bound, since $kt/s\geq1$ and
$\log(4k)=O(\log k)$ for $k\geq2$.
\end{proof}

\section{Random permutations}\label{sec:random}

\begin{proof}[Proof of Theorem~\ref{thm:random}]
Let $k$ be sufficiently large, and choose
\[
 s=\left\lceil\sqrt{k\log k}\right\rceil,
 \qquad h=\lceil16\log k\rceil.
\]
A cell with $s$ rows and $s$ columns has expected occupancy $s^2/k$,
which is asymptotic to $\log k$. We show that every cell has fewer
than $h$ ones with probability tending to one. For sufficiently large
$k$, we have $s\leq\lfloor k/2\rfloor+1$ and
$s^2/(k-h+1)\leq2\log k$. Fix a cell $I_a\times I_b$, and let
$Z_{ab}$ be its number of ones in $P_k$. If $Z_{ab}\geq h$, then some
$h$ rows in $I_a$ all map into $I_b$. Write
$(r)_h=r(r-1)\cdots(r-h+1)$. A union bound gives
\begin{align*}
 \Prb(Z_{ab}\geq h)
 &\leq\binom{s}{h}\frac{(s)_h}{(k)_h}\\
 &\leq\frac{s^{2h}}{h!(k-h+1)^h}
 \leq\left(\frac{2e\log k}{h}\right)^h
 \leq\left(\frac e8\right)^h,
\end{align*}
where we used $h!\geq(h/e)^h$. If either interval has fewer than
$h$ indices, the probability is zero, so the same bound holds.
There are at most $k^2$ cells. Another union bound gives
\[
 \Prb\bigl(t_s(P_k)\geq h\bigr)
 \leq k^2(e/8)^h
 \leq k^{\,2+16\log(e/8)}=o(1).
\]
With probability tending to one, we have $t_s(P_k)<h=O(\log k)$
and therefore $kt_s(P_k)/s=O(\sqrt{k\log k})$.
Theorem~\ref{thm:cells} gives
\[
 \log c(P_k)=O\bigl(\sqrt{k}(\log k)^{5/2}\bigr),
\]
with an absolute implied constant. By \eqref{eq:comparisons},
$\log L(P_k)\leq\log(2.88)+2\log c(P_k)$, so the same upper bound
holds for $\log L(P_k)$. This proves the theorem.
\end{proof}

\section{Concluding remarks}

For almost all permutations, the bounds on $\log c(P_k)$ in
\eqref{eq:lower} and Theorem~\ref{thm:random} differ by a factor
of order $(\log k)^3$. The same is true for $\log L(P_k)$ by
\eqref{eq:comparisons}. It would be interesting to determine
whether $\log c(P_k)=\Theta(\sqrt{k})$ with probability tending
to one. By \eqref{eq:comparisons}, this is equivalent to the same
question for $\log L(P_k)$.

\section*{Declaration of generative AI and AI-assisted technologies}
Codex with GPT-6 Astra Ultra assisted with finding the proofs,
checking the arguments, and revising the manuscript. The author
read and edited the proofs and takes responsibility for the
content of the article.

\end{document}